\documentclass[11pt]{article}
\usepackage[utf8]{inputenc}
\usepackage{amsmath, amssymb, amsthm}
\usepackage{fullpage}
\usepackage{graphics}
\usepackage{environ}
\usepackage{url}
\usepackage{xcolor,wrapfig}
\definecolor{cornellred}{rgb}{0.7, 0.11, 0.11}
\definecolor{dgreen}{rgb}{0.0, 0.5, 0.0}
\definecolor{ballblue}{rgb}{0.13, 0.67, 0.8}
\definecolor{royalblue(web)}{rgb}{0.25, 0.41, 0.88}
\definecolor{bleudefrance}{rgb}{0.19, 0.55, 0.91}
\definecolor{royalazure}{rgb}{0.0, 0.22, 0.66}

\usepackage{hyperref}
\hypersetup{
	colorlinks = true,
	linkcolor=royalazure,
	citecolor=royalazure,
	urlcolor=royalazure,
	linkbordercolor = {white}
}

\usepackage{algorithm}
\usepackage[noend]{algpseudocode}
\usepackage[labelfont=bf]{caption}
\usepackage{framed}
\usepackage[framemethod=tikz]{mdframed}
\usetikzlibrary{decorations.markings}
\usepackage{appendix}
\usepackage{graphicx}
\usepackage{mathtools}
\usepackage{enumitem}
\usepackage{subcaption}
\usepackage{cleveref}
\usepackage{csquotes}


\newcommand\remove[1]{}

\theoremstyle{plain}
\newtheorem{lemma}{Lemma}[section]

\newtheorem*{lemma*}{Lemma}

\newtheorem*{corollary*}{Corollary}

\newtheorem{claim}[lemma]{Claim}

\theoremstyle{definition}
\newtheorem{theorem}[lemma]{Theorem}
\newtheorem*{theorem*}{Theorem}
\newtheorem*{problem*}{Problem}
\newtheorem{definition}[lemma]{Definition}

\newtheorem{example}[lemma]{Example}

\newcommand\comment[1]{}

\newcommand\abs[1]{\vert{#1}\vert}
\newcommand\norm[1]{\Vert{#1}\Vert}

\newcommand\R{\mathbb{R}}

\newcommand\B{\textsf{B}}
\newcommand\F{\mathcal{F}}

\newcommand\Z{\mathbb{Z}}

\newcommand\T{\mathcal{T}}

\newcommand\E{\mathbb{E}}

\renewcommand\P{\mathbb{P}}
\newcommand\PP{\mathcal{P}}

\newcommand\eps{\varepsilon}

\renewcommand{\bf}{\textbf}

\newif\ifrandom
\randomtrue

\newcommand{\one}{{\bf 1}}

\usepackage[sort&compress,authoryear,round]{natbib}

\title{\bf{Multiparameter Bernoulli Factories}}
\author{Renato Paes Leme \\ Google Research \and Jon Schneider \\ Google Research}

\begin{document}
\maketitle

\begin{abstract}
We consider the problem of computing with many coins of unknown bias. We are given samples access to $n$ coins with \emph{unknown} biases $p_1,\hdots, p_n$ and are asked to sample from a coin with bias $f(p_1, \hdots, p_n)$ for a given function $f:[0,1]^n \rightarrow [0,1]$. We give a complete characterization of the functions $f$ for which this is possible. As a consequence, we show how to extend various combinatorial sampling procedures (most notably, the classic Sampford Sampling for $k$-subsets) to the boundary of the hypercube.

\end{abstract}

\section{Introduction}

The Bernoulli factory problem was formally introduced by  \cite{keane1994bernoulli}, inspired by earlier work by  \cite{von195113} and \cite{asmussen1992stationarity}. While their initial goal was to design methods to exactly simulate certain stochastic processes, this tool later found applications in many different fields such as mechanism design~(\citealt{dughmi2017bernoulli};~\citealt{cai2019efficient}), quantum physics~\citep{dale2015provable,yuan2016experimental}, Markov chain Monte Carlo (MCMC) methods~\citep{flegal2012exact}, and exact Bayesian inference~\citep{gonccalves2017exact,herbei2014estimating}.

The original problem can be best described as manufacturing new (random) coins from old ones. One is given a Bernoulli variable of \emph{unknown} bias $p$, or for short, a $p$-coin. Even though we do not know the bias, we can flip the coin as many times as we need obtaining i.i.d. samples from it. The goal is to produce an $f(p)$-coin for a given function $f:[0,1] \rightarrow [0,1]$.

To give an example, consider $f(p) = p^2 - p^3$. A simple way to sample from an $f(p)$-coin is to flip the $p$-coin three times obtaining samples $X_1, X_2, X_3 \in \{0,1\}$. Now, we output $1$ if $X_1=X_2=1$ and $X_3 =0$. The probability of outputting $1$ is $p^2(1-p) = f(p)$.

Keane and O'Brien gave necessary and sufficient conditions for a function $f:[0,1]\rightarrow [0,1]$ to be implementable. The first condition is that the function $f$ must be continuous. The second condition says that either $f$ is the constant function $f(x) = 0$, the constant function $f(x) = 1$, or there exists some integer $m > 0$ such that for all $p \in [0, 1]$:
\begin{equation}\label{eq:1d_cond}
\min(p,1-p)^m \leq f(p) \leq 1-\min(p,1-p)^m
\end{equation}
Furthermore, their proof is algorithmic: given a function satisfying the conditions above, they give a procedure for sampling from $f(p)$.

\subsection{Why exact sampling?} An important aspect of Bernoulli factories is that they ask for \emph{exact} sampling. The original motivation for the Bernoulli factory problem was to perform exact simulations of stochastic processes. In these simulations, small sampling errors quickly compound, sometimes exponentially -- hence the need for exact sampling.  A similar situation arises in Bayesian inference, where sampling is a sub-routine in an iterative procedure. 

Finally, in Mechanism Design the fact that sampling is exact allows us to design black-box-reductions that are Bayesian-incentive compatible. Before the introduction of this machinery, the known reduction in the general case was $\epsilon$-Bayesian-incentive-compatible, i.e. agents had still a small incentive to deviate from truth-telling, which results in a much weaker game-theoretical guarantee.

This discussion is to motivate why in certain situations \emph{approximately} simulating an $f(p)$-coin is not enough. Approximately sampling can be easily done by the following method: let $X_1, \hdots, X_t$ be $t$ draws from the $p$-coin and define its empirical average as
$\bar{X}_t = (X_1 + \hdots + X_t)/t$.
We know by the Chernoff bound that for $n = \Omega(\epsilon^{-2} \log(1/\delta))$ we have $\P[\abs{p - \bar{X}_t} > \epsilon] < \delta$. Hence if $f$ is continuous, estimating $p$ by $\bar{X}_t$ and using external randomness to sample from a $f(\bar{X}_t)$-coin produces a reasonable approximation of the $f(p)$-coin.

\subsection{Multiparameter Factories}
In this paper we study the multiparameter version of this problem: given a compact set $K \subseteq [0,1]^n$ and $n$ coins with \emph{unknown} biases $(p_1, \hdots, p_n)\in K$, how to sample from a $f(p_1,\hdots,p_n)$-coin for a multivariate function $f:K \rightarrow [0,1]$. Previous approaches to this problem either are restricted to rational functions (\cite{mossel2005new},  \cite{morina2019bernoulli} and \cite{NiazadehLS21}) or assume that the vector of coins $(p_1,\hdots,p_n)$ 
lies away from the boundary of the hypercube (\cite{nacu2005fast} and \cite{morina2021extending}).

Here we investigate how to design factories that terminate almost surely \emph{everywhere} on the domain $K$ and establish necessary and sufficient conditions for implementability.  While for the interior of the hypercube a single inequality suffices to check for implementability (\cite{morina2021extending}), when we require termination everywhere the conditions become combinatorial: one imposed by each open face of the hypercube. If the function is non-zero at an open face, it must be lower bounded by a polynomial associated with that face. Similarly, if the function is non-one at a face, it must be upper bounded by a polynomial associated with that face. See Definition \ref{def:poly_bounded} for a precise statement.

We show that these conditions also turn out to be sufficient (Theorem \ref{thm:main_thm}) and exhibit an algorithm to sample from it. The difficulty of designing such algorithm is to make sure it works near the faces of the hypercube and dealing with the interaction of multiple combinatorial constraints when the faces meet. The heart of the proof is a new concentration argument in Section \ref{subsec:fq_zero}. We study a random vector $\bar X_t$ where each coordinate is an average of Bernoulli variables drawn from the coins of unknown bias $p = (p_1, \hdots, p_n)$. We relate the probability of a large deviation in a subset $T$ of the coordinates to the combinatorial structure of the hypercube, in particular to the polynomials associated with the faces where coordinates in $T$ are free. By doing so, we can bound the probability that $f(\bar X_t) \geq 1/2$ in terms of $f(p)$ in a way which holds \textit{uniformly} everywhere within the hypercube (even on the boundary).

\subsection{The case of Sampford Sampling} A particularly curious Bernoulli factory is the procedure due to \cite{sampford1967sampling}. Sampford Sampling actually predates the notion of a Bernoulli factory by 25 years and is commonly used throughout the statistics literature for sampling $k$-subsets with ``unequal probabilities of selection''. Formally, the problem is the following: given probabilities $(p_1, \hdots, p_n)$ such that $\sum_i p_i = k$, sample a subset $S$ of size $k$ such that $\P[i \in S] = p_i$. Sampford's solution just requires sample access to the coins.


A natural but incorrect solution is the following: sample each coin $i$ once and let $X_i \in \{0,1\}$ be the outcome. Output $S = \{i \in [n]; X_i = 1\}$ if $\abs{S}=k$. If not, retry. To see that this does not work, execute this procedure with coins with biases $(1-\epsilon, 1-\epsilon, 2\epsilon)$ and observe that the last element is chosen with $O(\epsilon^2)$ probability. There is a simple (but ingenious) fix to this algorithm: first we obtain $S$ as before (retrying until $\abs{S} = k$). We then choose a coin in $[n]\setminus S$ uniformly at random and flip it again. If this coin comes up $1$, we output $S$. If not, we resample $S$ and try again.

For the previous procedure to terminate, we need at least one coin with $0 < p_i < 1$, since one of the coins that came up $0$ initially must be re-flipped and needs to come up $1$. If all the coins are deterministic, e.g. $p=(1,1,0)$, the above procedure never terminates.

A natural open question is whether there exists an alternative Bernoulli factory for this problem that terminates for every input in $p \in K := \{p \in [0,1]^n; \sum_i p_i=k\}$. \cite{NiazadehLS21} shows the following negative result: there is no \emph{exponentially-converging} factory for $k$-subset that terminates for all $p \in K$. A factory is said to be exponentially-converging if for every $p$, there is a rate $r(p)$ such that the probability that the procedure has not terminated after flipping $t$ coins is at most $r(p)^t$. Note that Sampford sampling is exponentially converging for every $p \in K \cap (0,1)^n$.

The negative result by \cite{NiazadehLS21} excludes techniques based on Bernstein-rational functions, which are the only known ideas for designing factories that terminate a.s. at the boundary but all lead to exponentially-converging factories. 

Despite this negative evidence, our new algorithm produces a factory for $k$-subset that terminates a.s. everywhere on $K$. More generally, it solves a wider class of problems introduced in \cite{NiazadehLS21}: given a polytope $\PP \subseteq [0,1]^n$ and $n$ coins with biases $p = (p_1,\hdots, p_n) \in \PP$, sample a random vertex $v$ of $\PP$ such that $\E[v] = p$. \cite{NiazadehLS21} show that is is possible to construct factories for $\PP \cap (0,1)^n$ only when $\PP$ is the intersection of the hypercube $[0,1]^n$ and an affine subspace.

These factories, however, suffer from the same problem as Sampford sampling: they diverge for certain points in the boundary of the hypercube. Using the techniques developed in this paper, we exhibit alternative factories that terminate a.s. for all points in $\PP$. For example, consider the matching polytope: we are given coins $p_{ij}$ forming a doubly stochastic matrix and asked to sample a matching $M$ in the complete bipartite graph such that $\P[(i,j) \in M] = p_{ij}$. The previous factory required $p_{ij}>0$ for all edges $(i,j)$. The alternative factories constructed in this paper no longer have this restriction.

\subsection{Previous results on Multiparameter Factories}
\cite{mossel2005new} and  \cite{morina2019bernoulli} show how to design factories for Bernstein-rational functions, i.e., rational functions of the type $f(p) = a(p) / [a(p)+b(p)]$ where $a(p)$ and $b(p)$ are of the form $\sum_i c_i \prod_{j \in [n]} p_i^{a_{ij}} (1-p_i)^{b_{ij}}$ with $c_i > 0$. Recently, \cite{NiazadehLS21} showed how to design factories for certain combinatorial objects (such as matchings and flows) using Bernstein-rational functions. 

Beyond rational functions, \cite{nacu2005fast} give a procedure for sampling from any continuous function $f:[\epsilon, 1-\epsilon]^n \rightarrow (0,1)$. Their procedure is based on Bernstein's proof the Weierstrass approximation theorem. Their result is originally written for $n=1$ but there is nothing particular about one dimension in their construction.

Using a very clever idea, \cite{morina2021extending} shows in Chapter 3 of his PhD thesis how to combine factories defined in $[\epsilon,1-\epsilon]^n$ for decreasing values of $\epsilon$ into a single factory defined on the interior of the hypercube $(0,1)^n$. The condition required for the factories to be combined is a generalization of the condition of Keane and O'Brien. The result is stated for the open simplex $\Delta_n^0 = \{p \in (0,1)^n; \sum_{i=1}^n p_i = 1\}$. It shows that a function $f:\Delta_n^0 \rightarrow (0,1)$ is implementable by a Bernoulli factory if and only if it is continuous and for some integer $m$ it holds that:
$$\textstyle \left(\prod_i p_i\right)^m < f(p) < 1- \left(\prod_i p_i\right)^m, \forall p \in \Delta_n^0$$

The problem with Morina's factory is that it diverges by construction at the boundary and its running time blows up when we get close to it. The first step in its construction (Lemma 3.13 in  \cite{morina2021extending}) is to keep sampling all of the coins until each coin comes up $1$ at least $\Omega(mn)$ times and $0$ at least $\Omega(mn)$ times. At a high level, it uses the coin outcomes to pick a value of $\epsilon$ and then it uses a factory for $[\epsilon,1-\epsilon]^n$ domain. The process never terminates if the input has coins with $p_i \in \{0,1\}$. 

\section{Preliminaries}

\subsection{Multiparameter factory}
We start by giving a formal definition of a multiparameter factory (following \cite{NiazadehLS21}):

\begin{definition}\label{def:factory}
A Bernoulli factory $\F$ with input $(p_1,\hdots, p_n)$ corresponds to a (possibly infinite) rooted binary tree $\T$ where each node of $\T$ has either $2$ children (an internal node) or $0$ (a leaf). Each internal node is labelled either with a variable $p_i$ or with a constant $c \in (0,1)$. Each leaf is labelled with $0$ or $1$, representing the output of the factory upon reaching that node.

To execute the factory with coins $(p_1, \hdots, p_n)$ we start from the root and at each node we flip the coin given the label of that node (either one of the $p_i$-coins of unknown bias or a $c$-coin of known bias). Based on the outcome, we either take the left edge ($0$) or the right edge ($1$). If we reach a leaf, we output its label.
\end{definition}

A factory $\F$ is valid if for any input $p = (p_1, \hdots, p_n) \in [0,1]^n$ it reaches a leaf almost surely. Therefore we can view a factory as a random variable $\F$ taking values in $\{0, 1\}$. The distribution of $\F$ will naturally depend on the input coins $p$. For that reason it is convenient to use the notation $\P_p[\cdot]$ and $\E_p[\cdot]$ to denote the probability measure and expectation of random variables when $(p_1,\hdots, p_n)$ coins are used.

We will say that a factory is finite if the tree $\T$ contains finitely many nodes (and thus, the factory is guaranteed to terminate after a finite number of coin flips).

\subsection{Concentration Bounds}

We will use $X_t \in \{0,1\}^n$ to denote i.i.d. random variables corresponding to the input coin flips. For each $t = 1,2,\hdots$ and $i\in [n]$, the variable $X_{t,i}$ is an independent Bernoulli variable with bias $p_i$. We will also let $\bar X_t$ be a random variable equal to the average of the first $t$ flips of all $n$ coins:
\begin{equation}
\bar X_t := \frac{X_1 + \hdots + X_t}{t} \in [0,1]^n
\end{equation}

We will write $\bar X_{t,i}$ to denote the $i$-th component of $\bar X_t$.
Next, we state two well known concentration bounds. The first is the Hoeffding bound:
\begin{equation}\label{eq:hoeffding}
    \P_p[\abs{\bar X_{t,i} - p_i} > \delta] \leq 2 \exp(-2\delta^2 t)
\end{equation}
The second is the sharper Chernoff bound:
\begin{equation}\label{eq:chernoff0}
    \P_p[\bar X_{t,i} - p_i > \delta] \leq \left(\left( \frac{p_i}{p_i+\delta}\right)^{p_i+\delta} \left( \frac{1-p_i}{1-p_i-\delta}\right)^{1-p_i+\delta}\right)^t
\end{equation}
For values of $p_i$ that are closer to zero (say $p_i < 1/2$ and $\delta < 1/4$) we can bound the second term by a constant. Hence for such $p_i$ we can write:
\begin{equation}\label{eq:chernoff}
    \P_p[\bar X_{t,i} - p_i > \delta] \leq \left( c_\delta \cdot p_i^\delta\right)^t
\end{equation}
where $c_\delta$ is some constant depending on $\delta$.

\subsection{Real Topology}

Here we recall some elementary facts and definitions from real topology. For $p \in \R^n$ and a real $r >0$, we denote the $\ell_\infty$-ball around $p$ of radius $r$ by:
$$\B_\infty(p;r) := \{ x \in \R^n; \norm{x-p}_\infty <r\}$$
Given any set $S \subseteq \R^n$ we denote:
$$\B_\infty(S;r) := \cup_{p \in S} \B_\infty(p;r) = \{ x \in \R^n; \exists\, p \in S \text{ s.t. } \norm{x-p}_\infty <r\}$$
A set $U \subset \R^n$ is open if for every $p \in U$ there is an $r>0$ such that $\B_\infty(p;r) \subseteq U$. An \emph{open cover} of a set $S$ consists of a collection of open sets $U_i$ for $i \in I$ such that $\cup_{i \in I} U_i \supseteq S$. The index set $I$ is possibly infinite and uncountable. We say that the cover $\cup_{i \in I} U_i \supseteq S$ admits a \emph{finite subcover} if there is a finite set $I_0 \subset I$ such that $\cup_{i \in I_0} U_i \supseteq S$.

We will make extensive use of the following elementary fact:


\begin{lemma}[Heine-Borel]\label{lemma:compactness}
A set $K$ is a compact set (i.e. every cover admits a finite subcover) iff it is topologically closed and bounded.
\end{lemma}

\subsection{Decomposing the hypercube}

We will decompose the hypercube $[0,1]^n$ into $3^n$ disjoint regions that we will refer as open faces. Each open face will correspond to a partition of $[n] := \{1, \hdots, n\}$ into three sets $A$,$S$ and $B$. We define the open face $F_{A,S,B}$ as:
$$F_{A,S,B} := \left\{p \in [0,1]^n; \quad
\begin{aligned} & p_i = 0, & &i \in A\\
& 0< p_i <1, & &i \in S\\
& p_i = 1, & &i \in B\\ \end{aligned} \right\}$$

For example, the square $[0,1]^2$ is the union of $9$ disjoint open faces: the interior $(0,1)^2$, the four edges $\{0\} \times (0,1), \{1\} \times (0,1), (0,1) \times \{0\}, (0,1) \times \{1\}$ and the four vertices $\{(0,0)\}, \{(0,1)\}, \{(1,0)\}, \{(1,1)\}$.

We will denote by $\bar F_{A,S,B}$ the topological closure of $F_{A,S,B}$ which can be written as: 
$$\bar F_{A,S,B} = \bigcup_{A \subseteq A', B \subseteq B'} F_{A',S',B'}$$

For example, the closure of the open face $\{0\} \times (0,1)$ of the square is: $\{0\} \times [0,1]$ which is the union of three open faces: one representing that edge and two representing the vertices.

\subsection{Additional Notation}

We will be concerned with functions $f:[0,1]^n \rightarrow [0,1]$. Given a subset $S \subseteq [0,1]^n$ and a constant $c\in[0,1]$ we will write $f\vert_S \equiv c$ to denote that $f(p) = c$ for all $p \in S$. If we write $f \equiv c$ (omitting $S$) it means $f$ is the constant function on the entire hypercube.\\

\noindent We will use $[n]$ to denote $\{1,2, \hdots, n\}$. Given a vector $p \in [0,1]^n$ and $S \subseteq [n]$ we will denote:
$$p^S = \prod_{i \in S} p_i \quad \text{and} \quad (1-p)^S = \prod_{i \in S} (1-p_i)$$

\section{Main Theorem}

Our main result is to identify a condition called polynomially-boundedness which together with continuity is necessary and sufficient for the existence of a factory. We define it below:

\begin{definition}[polynomially-bounded function]\label{def:poly_bounded}
A function $f : [0,1]^n \rightarrow [0,1]$ is polynomially bounded if there is an integer $m \geq 0$ and a real constant $c > 0$ such that for each open face $F_{A,S,B}$ of the hypercube the following condition holds:
\begin{equation}\label{eq:poly_bounded_zero}
f \vert_{F_{A,S,B}} \not\equiv 0 \Rightarrow f(p) \geq c \left(  (1-p)^A \cdot p^S (1-p)^S \cdot p^B \right)^m, \forall p \in [0,1]^n
\end{equation}
\end{definition}

Now, we are ready to state the main theorem:

\begin{theorem}\label{thm:main_thm}
A function $f:[0,1]^n \rightarrow [0,1]$ can be implemented by a Bernoulli factory if and only if it is continuous and both $f$ and $1-f$ are polynomially-bounded.
\end{theorem}
 
\subsection{Sanity Check}

It is useful to check that when we set $n=1$ we recover the result by Keane and O'Brien. For $n=1$ there are $3$ open faces: $\{0\}$, $\{1\}$ and $(0,1)$. For $(0,1)$ the condition that $f$ is polynomially bounded means that:
$$f \not\equiv 0 \Rightarrow f(p) \geq c (p(1-p))^m$$
Observe that $\min(p, 1-p) \leq \frac{1}{2}$, so if we take $k = \lceil \log_2 c \rceil$ then: $c p(1-p)^m \geq \min(p,1-p)^{k+2m}$. Similarly the condition that $1-f$ is polynomially bounded implies that:
$$f \not\equiv 1 \Rightarrow f(p) \leq 1- c (p(1-p))^m$$
Hence if $f \not\equiv 0$ and $f \not\equiv 1$ then:
$$ \min(p,1-p)^{k+2m} \leq f(p) \leq 1 - \min(p,1-p)^{k+2m}$$
which is the one-dimensional condition \eqref{eq:1d_cond}. Notice that the conditions for the open faces $\{0\}$ and $\{1\}$ are superfluous here. For example, for $\{0\}$ our condition says that:
$$f(0)>0 \Rightarrow f(p) \geq c (1-p)$$
Note that this is implied by continuity in a neighborhood of $0$ and by the condition \eqref{eq:1d_cond} elsewhere.

\subsection{Necessary conditions}

The next lemmas show that every function $f$ that is implementable by a factory has $f$ and $1-f$ polynomially bounded.

\begin{lemma}\label{lemma:necessary1} If a function $f : [0,1]^n \rightarrow [0,1]$ can be implemented by a Bernoulli factory and for some open face $F_{A,S,B}$ we have $f\vert_{F_{A,S,B}} \not\equiv 0$ then there exists an integer $m$ and a constant $c$ such that $f(p) \geq c \left(  (1-p)^A \cdot p^S (1-p)^S \cdot p^B \right)^m, \forall p \in [0,1]^n$.
\end{lemma}

\begin{proof}
Consider a rooted binary tree (as in Definition \ref{def:factory}) implementing $f$. Since $f(p) > 0$ for some coins in $p \in F_{A,S,B}$, there must be a path in the tree reaching a leaf labelled $1$ that always takes the $0$-edge when we flip a coin with an index in $A$ and always takes the $1$-edge when we flip a coin with an index in $B$, or else this path would never be reachable using the coins $p$. Each path in the tree corresponds to a polynomial of the form $c \cdot \prod_{i \in [n]} p_i^{g_i} (1-p_i)^{h_i}$ (a ``Bernstein monomial''), where $c$ is the product of the helper coins flipped along the path, $g_i$ is the number of $0$-edges takes after a $p_i$-flip and $h_i$ is the number of $1$-edges takes after a $p_i$-flip. By the observation above $g_i = 0$ for $i \in A$ and $h_i = 0$ for $i \in B$. Taking $m = \max_i \max(g_i, h_i)$ we obtain the inequality in the statement of the lemma.
\end{proof}

\begin{lemma}\label{lemma:necessary2} If a function $f : [0,1]^n \rightarrow [0,1]$ can be implemented by a Bernoulli factory and for some open face $F_{A,S,B}$ we have $f\vert_{F_{A,S,B}} \not\equiv 1$ then there exists an integer $m$ and a constant $c$ such that $1-f(p) \geq c \left(  (1-p)^A \cdot p^S (1-p)^S \cdot p^B \right)^m, \forall p \in [0,1]^n$.
\end{lemma}

\begin{proof}
Same proof as the previous lemma swapping $0$ and $1$.
\end{proof}

The continuity condition is more intuitive: if two vectors of coins $p',p'' \in [0,1]^n$ are close, the finite sequence of coin flips generated by them will also be close in total variation distance. Since the output only depends on the sequence of coins flips observed, the distribution of outputs must also be close. This intuition is formalized by the following lemma, whose proof can be found in the appendix.

\begin{lemma}\label{lemma:necessary3}
If a function $f : [0,1]^n \rightarrow [0,1]$ can be implemented by a Bernoulli factory then it is continuous.
\end{lemma}

\subsection{Sufficient conditions}\label{subsec:sufficient}
 
To prove that continuous and polynomially-bounded are sufficient conditions for implementability, we will use the following lemma (Lemma \ref{lemma:main_lemma}), which is the main technical lemma of the paper. We will prove it in the next section. Before we do this, however, we will assume it is true and use it to prove Theorem \ref{thm:main_thm}.
 
\begin{lemma}\label{lemma:main_lemma}
Let $f$ be a continuous and polynomially bounded function.
Then there is an integer $t_0$ such that for $t \geq t_0$ it holds that:
\begin{equation}\label{eq:main_lemma}
f(p) - \frac{1}{4} \cdot \P_p \left[ f(\bar X_t) \geq \frac{1}{2} \right] \geq \frac{1}{8} f(p), \forall p \in [0,1]^n
\end{equation}
\end{lemma}

Lemma \ref{lemma:main_lemma} will allow us to decompose $f(p)$ into two smaller functions: one which we can simulate with a finite Bernoulli factory (this will be $\P_p \left[ f(\bar X_t) \geq \frac{1}{2} \right]$), and a remaining piece with probability mass at most $3/4$ that we can decompose recursively. 

\begin{lemma}\label{lemma:recursive_step}
Let $f : [0,1]^n \rightarrow [0,1]$ be a continuous function such that $f$ and $1-f$ are polynomially bounded. Then there exists a function $g:[0,1]^n \rightarrow [0,1]$ that is implementable by a finite Bernoulli factory such that $\tilde{f}$ defined as follows:
$$\tilde{f}(p) = \frac{4}{3}\left( f(p) - \frac{1}{4} g(p) \right)$$
maps $[0,1]^n$ to $[0,1]$, is continuous and $\tilde{f}$ and $1-\tilde{f}$ are polynomially bounded.
\end{lemma}

\begin{proof}
We start by applying Lemma \ref{lemma:main_lemma} to both $f$ and $1-f$. We know that there are integers $t_0$ and $t_1$ such that:
$$f(p) - \frac{1}{4} \cdot \P_p \left[ f(\bar X_t) \geq \frac{1}{2} \right] \geq \frac{1}{8} f(p), \forall t \geq t_0, p \in [0,1]^n$$
$$1-f(p) - \frac{1}{4} \cdot \P_p \left[ 1-f(\bar X_t) \geq \frac{1}{2} \right] \geq \frac{1}{8} \left[1-f(p)\right], \forall t \geq t_1, p \in [0,1]^n$$
Note that we can rewrite:
$$\P_p \left[ 1-f(\bar X_t) \geq \frac{1}{2} \right] =
\P_p \left[ f(\bar X_t) \leq \frac{1}{2} \right] = 1-\P_p \left[ f(\bar X_t) > \frac{1}{2} \right] \geq 1-\P_p \left[ f(\bar X_t) \geq \frac{1}{2} \right]$$
Replacing it in the previous expression we obtain:
$$f(p) - \frac{1}{4} \P_p \left[ f(\bar X_t) \geq \frac{1}{2} \right] \leq \frac{3}{4} - \frac{1}{8} [1-f(p)], \forall t \geq t_1, p \in [0,1]^n$$
Now, if we set $t = \max(t_0, t_1)$ and define a function $g:[0,1]^n \rightarrow [0,1]$ as:
$$g(p) = \P_p \left[ f(\bar X_t) \geq \frac{1}{2} \right]$$
then we have that:
\begin{equation}\label{eq:poly_bounds}
\frac{1}{6} f(p) \leq \frac{4}{3}\left( f(p) - \frac{1}{4} g(p) \right) \leq 1- \frac{1}{6}[1-f(p)]
\end{equation}
First observe that \eqref{eq:poly_bounds} implies that $\tilde{f}(p) \in [0,1]$.

Now, let's argue that $\tilde{f}$ is polynomially bounded. First observe that if $f(p) = 0$ at some point $p$ then we must have $\tilde{f}(p) = 0$ since the first inequality implies that $0 \leq - \frac{1}{4}g(p)$. Since $g(p) \geq 0$ we must have $g(p) = 0$ and hence $\tilde f(p) = \frac{4}{3} f(p) = 0$. 

Therefore, for any open face $F_{A,S,B}$ we have $f \vert_{F_{A,S,B}} \equiv 0$ iff $\tilde f \vert_{F_{A,S,B}} \equiv 0$. If $\tilde f \vert_{F_{A,S,B}} \not\equiv 0$ then by the first inequality in \eqref{eq:poly_bounds} and the fact that $f$ is polynomially bounded, we have:
$$\tilde{f}(p) \geq \frac{1}{6} f(p) \geq \frac{c}{6} ((1-p)^A p^S (1-p)^S p^B)^m$$

The same argument can be repeated with $1-f$ instead of $f$ to argue this function is also polynomially bounded. First observe that if $f(p) = 1$ at some point $p$ we must have $\tilde{f}(p) = 1$ since the last inequality would imply that $1-\frac{1}{4} g(p) \leq \frac{3}{4}$. Since $g(p) \leq 1$ this must imply that $g(p) = 1$ and hence $\tilde f(p) = \frac{4}{3}(1-\frac{1}{4})=1$. Therefore $f \vert_{F_{A,S,B}} \equiv 1$ iff $\tilde f \vert_{F_{A,S,B}} \equiv 1$. If $\tilde f \vert_{F_{A,S,B}} \not\equiv 1$ then by the second inequality in \eqref{eq:poly_bounds} and the fact that $1-f$ is polynomially bounded, we have:
$$1-\tilde{f}(p) \geq \frac{1}{6} [1-f(p)] \geq \frac{c}{6} ((1-p)^A p^S (1-p)^S p^B)^m$$
The only part left to argue is that $g$ is implementable by a finite Bernoulli factory, but this follows by the definition of $g$: one can implement it by taking $t$ samples of each coin, building $\bar X_t$ and checking whether $f(\bar X_t) \geq \frac{1}{2}$.
\end{proof}

We now can derive the proof of Theorem \ref{thm:main_thm} by recursively applying the previous lemma:

\begin{proof}[Proof of Theorem \ref{thm:main_thm}]
We will define a sequence of functions $f_1(p), f_2(p), \hdots$ as follows. First we define $f_1(p) = f(p)$. Then for every $k \geq 1$ let $g_k$ correspond to the $g$ function in Lemma \ref{lemma:recursive_step} obtained from $f_k$. Then define $f_{k+1}(p) = \frac{4}{3}\left( f_k(p) - \frac{1}{4} g_k(p) \right)$. Unrolling the recursion we get:
$$f(p) = \left( \frac{3}{4} \right)^k f_{k+1}(p) + \sum_{s=1}^k \frac{1}{4} \left( \frac{3}{4} \right)^{s-1} g_s(p)$$
Since $f_{k+1}(p) \in [0,1]$, it means that:
$0 \leq f(p) - \sum_{s=1}^k \frac{1}{4} \left( \frac{3}{4} \right)^{s-1} g_s(p) \leq \left( \frac{3}{4} \right)^k, \forall p \in[0,1]^n $
or in other words, the series
$\sum_{k=1}^\infty \frac{1}{4} \left( \frac{3}{4} \right)^{k-1} g_k(p)$
converges uniformly to $f(p)$. This observation gives a natural algorithm for sampling from $f(p)$: first sample an index $k \in \Z_{>0}$ with probability $\frac{1}{4} \left( \frac{3}{4} \right)^{k-1}$. Then use the Bernoulli factory for $g_k(p)$ constructed in Lemma \ref{lemma:recursive_step} to sample $1$ with that probability. 
\end{proof}

\section{Proof of Lemma \ref{lemma:main_lemma}}

The heart of the argument is establishing Lemma \ref{lemma:main_lemma}. Note that as $t \rightarrow \infty$ we know that 
$$\P_p\left[f(\bar X_t) \geq \frac{1}{2}\right] \rightarrow \one\left\{f(p) > \frac{1}{2} \right\}$$
if $f(p) \neq \frac{1}{2}$, so clearly it holds that for each $p \in [0,1]^n$ there is a large enough $t$ such that the inequality in the lemma holds. The difficulty in the argument is to show a single $t$ holds for all points $p$ simultaneously.

The argument will proceed as follows: first let's define the region where the values of $f$ are small:
$$L = \left\{ p \in [0,1]^n; f(p) \leq \frac{3}{8} \right\}$$

It is simple to see that if $p \notin L$ then for any value of $t$ it holds that:
$$f(p) - \frac{1}{4} \P_p\left[f(\bar X_t) \geq \frac{1}{2}\right] \geq \frac{3}{8} - \frac{1}{4} = \frac{1}{8}  \geq \frac{1}{8} f(p)$$

\paragraph{Proof strategy} To argue that there exists some $t$ that holds for all $q \in L$, we will prove the following claim:

\begin{claim}\label{claim:ball}
For every $q \in L$ there is a radius $r_q>0$ and an integer $t_q$ such that equation \eqref{eq:main_lemma} in the statement of Lemma \ref{lemma:main_lemma} holds for all $p \in L \cap \B_\infty(q;r_q)$ and $t \geq t_q$.
\end{claim}

If Claim \ref{claim:ball} is true, then it provides us with an open cover $\cup_{q \in L} \B_\infty(q;r_q)$ of $L$. Hence we can use Lemma \ref{lemma:compactness} to argue it must have a finite subcover, i.e., there is a finite set of points $Q \subset L$, $\abs{Q} < \infty$ such that $L \subset \cup_{q \in Q} \B_\infty(q;r_q)$. Hence if we take $t_0 = \max_{q \in Q} t_q$, then equation \eqref{eq:main_lemma} in Lemma \ref{lemma:main_lemma} holds for all $t \geq t_0$ and $p \in L$.

\subsection{A safe distance from the boundary}

We will start by making two observations about the geometry of set $L$. 

\begin{claim}\label{claim:buffer}
There is some positive constant $\delta > 0$ such that:
$$f(q) < \frac{1}{2}, \forall q \in \B_\infty(p,\delta) \text{ and } p \in L.$$
\end{claim}
\begin{proof}
Define $H = \left\{ p \in [0,1]^n; f(p) \geq \frac{1}{2} \right\}$. Since $H$ and $L$ are disjoint compact sets there is a constant $\delta$ such that  $\norm{p-q}_\infty > \delta, \forall p \in H, q \in L$. 
\end{proof}

With that observation, for any $p \in L$ we will bound 
$\P_p\left[f(\bar X_t) \geq \frac{1}{2}\right] \leq \P_p\left[ \norm{\bar X_t - p}_\infty > \delta \right]$,
which allows us to apply concentration bounds. This will give us a good enough argument whenever $f(p) > 0$. For the case $f(p) = 0$, however, we need a more detailed understanding of how $f$ behaves close to the boundary. For that, we will use the following two claims:

\begin{claim}\label{claim:zero_open_face}
If $p \in F_{A,S,B}$ and $f(p) = 0$ then $f\vert_{F_{A,S,B}} \equiv 0$.
\end{claim}

\begin{proof}
If $f\vert_{F_{A,S,B}} \not\equiv 0$ then $f(q) \geq c \cdot ((1-q)^A q^S (1-q)^S q^B)^m$ which contradicts the fact that $f(p) = 0$ since $(1-p)^A = 1$, $p^B = 1$ and $p^S (1-p)^S > 0$.
\end{proof}

The second claim establishes that if a function is zero on an open face, then there exists a safe distance $\delta$ such that any point $q$ within distance $\delta$ of this open face satisfies $f(q) < 1/2$.

\begin{claim}\label{claim:safe_distance}
If $f\vert_{F_{A,S,B}} \equiv 0$ then $f(q) < 1/2$ for all $q \in [0,1]^n \cap B_\infty(F_{A,S,B}, \delta)$ for the constant $\delta$ in Claim \ref{claim:buffer}.
\end{claim}

\begin{proof}
Follows directly from Claim \ref{claim:buffer} and the fact that $F_{A,S,B} \subset L$.
\end{proof}

\subsection{Proof of Claim \ref{claim:ball} for $f(q) > 0$}\label{subsec:fq_pos}

By continuity, there is a radius $r$ small enough such that:
$$2 f(q) \geq f(p) \geq \frac{1}{2} f(q), \forall p \in L \cap \B_\infty(q,r)$$
By Claim \ref{claim:buffer} and the Hoeffding bound (equation \eqref{eq:hoeffding}) we know that:
$$
\P_p\left[ f(\bar X_t) \geq \frac{1}{2}\right] \leq
\P_p\left[\norm{\bar X_{t} - p}_\infty > \delta\right] \leq 2n \exp(-2\delta^2 t) \leq \frac{1}{4} f(q)
$$
for $t \geq t_q := \lceil -\frac{1}{2 \delta^2}\log(\frac{1}{8n} f(q))\rceil$ and $p \in L$.
Therefore, we have:
$$f(p) - \frac{1}{4} \P_p\left[f(\bar X_t) \geq \frac{1}{2} \right] \geq \frac{1}{2} f(q) - \frac{1}{4} f(q) = \frac{1}{4} f(q) \geq \frac{1}{8} f(p)$$
therefore establishing equation \eqref{eq:main_lemma} for $p \in L \cap \B_\infty(q,r) $ and $t \geq t_q$.

\subsection{Proof of Claim \ref{claim:ball} for $f(q) = 0$}\label{subsec:fq_zero}

Let $F_{A,S,B}$ be the open face containing $q$. By Claim \ref{claim:zero_open_face} we know that $f\vert_{F_{A,S,B}} \equiv 0$.  We start by taking a small radius $r$ such that $r < \delta/2$ and $q_i - r > r > 0$ and $1-q_i - r > r > 0$ for all $i \in S$. Through the course of the proof, we may decrease $r$ further if necessary.

Our goal is to bound the probability $\P_p[f(\bar X_t) \geq 1/2]$ for $p \in \B_\infty(q,r) \cap L$. We will decompose this probability by looking at the coordinates (specifically, the coordinates in $A \cup B$) on which $\bar{X}_t$ significantly deviates from $p$:
$$\P_p\left[f(\bar X_t) \geq \frac{1}{2}\right] = \sum_{T \subseteq A \cup B} \P_p\left[f(\bar X_t) \geq \frac{1}{2} \text{ and } T = \left\{i \in A \cup B; \abs{\bar X_{t,i} - p_i} > \frac{\delta}{2}\right\}\right]$$
We will show that for each $T \subseteq A \cup B$ there is $r$ small enough such that:
\begin{equation}\label{eq:bound_T}
\P_p\left[f(\bar X_t) \geq \frac{1}{2} \text{ and } T = \left\{i \in A \cup B; \abs{\bar X_{t,i} - p_i} > \frac{\delta}{2}\right\}\right] \leq \frac{1}{2^{n+2}} f(p)
\end{equation}
If we can establish this, then it would imply that $\P_p\left[f(\bar X_t) \geq \frac{1}{2}\right] \leq \frac{1}{4} f(p)$ which directly implies equation \eqref{eq:main_lemma}.
We will consider two cases depending the value of $f$ in the open face $F_{A \setminus T, S \cup T, B \setminus T}$.\\

\noindent \emph{Case 1:} $f_{F_{A\setminus T, S\cup T, B \setminus T}} \equiv 0$. In this case, by Claim \ref{claim:safe_distance} we have that  all points $p$ at a $\ell_\infty$ distance at most $\delta$ from $F_{A\setminus T, S\cup T, B \setminus T}$ have $f(p)<1/2$. Now observe that all the coordinates in which $p$ differs from $\bar X_t$ by at least $\delta/2$ belong to $S \cup T$. Hence, the $\ell_\infty$ distance between $\bar X_t$ and $F_{A\setminus T, S\cup T, B \setminus T}$ is at most $\delta/2$, and hence $\bar X_t \in L$, which in turn means that $f(\bar X_t) < \frac{1}{2}$. So the left hand side of equation \eqref{eq:bound_T} is zero.\\

\noindent \emph{Case 2:} $f_{F_{A\setminus T, S\cup T, B \setminus T}} \not\equiv 0$. In this case, by the fact that $f$ is polynomially bounded we have:
$$f(p) \geq c \cdot ((1-p)^{A \setminus T} \cdot p^{S \cup T} (1-p)^{S \cup T} \cdot p^{B\setminus T})^m $$
For $p \in \B_\infty(q,r)$ we have $1-p_i > r$ for $i \in A \cup S$ and $p_i > r$ for $i \in S \cup B$ by the definition of $r$. Since $r$ is a constant we can write:
$$f(p) \geq C \cdot (p^{T\cap A} \cdot(1-p)^{T \cap B})^m$$
Now, note that the left hand side of equation \eqref{eq:bound_T} is at most:
$$\P_p\left[ \bar X_{t,i} - p_i > \frac{\delta}{2}, \forall i \in T \cap A \text{ and }  p_i - \bar  X_{t,i}> \frac{\delta}{2}, \forall i \in T \cap B \right]$$
since for $i \in A$ we have $p_i - \bar X_{t,i} \leq p_i < r < \delta/2$ and for $i \in B$ we have $\bar X_{t,i} - p_i \leq 1-p_i \leq r < \delta/2$.
By the Chernoff bound (equation \eqref{eq:chernoff0}) this can be bounded by:
$$\prod_{i \in A\cap T}(C' p_i^{\delta/2})^t \cdot \prod_{i \in B\cap T}(C' (1-p_i)^{\delta/2})^t = (C')^{\abs{T} t} \cdot (p^{A \cap T} (1-p)^{B \cap T})^{\delta t/2}$$
for some constant $C'$. Choose $t = \lceil 4m/\delta\rceil$. Then for a constant $C''$ we can bound the expression above by:
$$C''\cdot(p^{A \cap T} (1-p)^{B \cap T})^{2m} $$
We know that $\abs{T} \geq 1$ since $f\vert_{A,S,B} \equiv 0$, so $p^{A \cap T} (1-p)^{B \cap T} \leq r$. It follows that for sufficiency small values of $r$:
$$r \leq \left(\frac{C}{2^n C''}\right)^{1/m}$$
and therefore equation \eqref{eq:bound_T} holds for all $p \in L \cap \B_\infty(q,r)$.

\section{Factories on Subdomains of the Hypercube}\label{sec:subdomain}

Next we characterize which functions defined on a compact subdomain of the hypercube are implementable. Given a compact set $K \subseteq [0,1]^n$ and a continuous function $f:K \rightarrow [0,1]$, our first instinct is to check whether we can extend $f$ to the entire hypercube satisfying the properties in Theorem \ref{thm:main_thm}. The problem with this approach is that it is possible to construct functions that are implementable on a subset $K$ but can't be extended to an implementable function on the entire hypercube. Here is an example:

\begin{example}
Let $K$ be the convex hull of $\{(\frac{1}{2},0), (0, \frac{1}{2})\}$ and define $f:K \rightarrow [0,1]$ as $f(p) = p_1 / (p_1+p_2)$. We first observe that $f$ is implementable by the following procedure: choose a coin $i \in \{1,2\}$ uniformly at random and flip it. If the coin comes up $1$, then output $1$ if $i=1$ and $0$ if $i=2$. If the flipped coin comes up $0$ we retry. Each time we do it, we output $1$ with probability $\frac{p_1}{2}$, we output $0$ with probability $\frac{p_2}{2}$ and retry with probability $1-\frac{p_1+p_2}{2}$. The total probability of outputting one is therefore $\sum_{k=0}^\infty \left( 1-\frac{p_1+p_2}{2} \right)^k \frac{p_1}{2} = \frac{p_1}{p_1 + p_2}$.

However, $f$ doesn't have a continuous and polynomially bounded extension to the hypercube. To see that observe that since $f(\frac{1}{2},0) = 1$, then it must be $1$ on the open face $(0,1)\times \{0\}$ by the fact it is polynomially bounded. Similarly, since $f(0,\frac{1}{2}) = 0$ then it must be $0$ on the open face $\{0\} \times (0,1)$. Such a function can't be continuous on $[0,1]^2$ since there are sequences approaching $(0,0)$ with different limits.
\end{example}

Instead of trying to extend the function we will adapt the proof in the previous section to deal with any domain $K$. First we say that a function $f:K \rightarrow [0,1]$ is polynomially bounded if there is an integer $m$ and a real constant $c>0$ such that for every open face $F_{A,S,B}$ of the hypercube, it holds that:
$$\exists q \in K \cap F_{A,S,B}, f(q) > 0 \Rightarrow f(p) \geq c  ((1-p)^A \cdot p^S \cdot (1-p)^S \cdot p^B)^m, \forall p \in K$$

With this extended definition we state the following:

\begin{theorem}[Extension of Theorem \ref{thm:main_thm} to subdomains]\label{thm:sub_main_thm} For a compact $K \subseteq [0,1]^n$, a function $f:K \rightarrow [0,1]$ is implementable by a Bernoulli factory if and only if it is continuous and $f$ and $1-f$ are polynomially bounded.
\end{theorem}

The necessary conditions follow from the exact same arguments as in Lemmas  \ref{lemma:necessary1}, \ref{lemma:necessary2} and \ref{lemma:necessary3}. To prove it is sufficient, we need to deal with the following difficulty in extending the proof: the sampled average $\bar X_t$ may not be in the domain $K$ so $f(\bar X_t)$ is not well-defined. Therefore we can't use Lemma \ref{lemma:main_lemma} directly.

To address this, we need two new ideas. The first new idea is to project the sampled point $\bar X_t$ to the domain $K$. One difficulty in simply projecting the sampled average is that if $Y$ is the projection of $\bar X_t$ to $K$ then a large deviation in one coordinate (say $\abs{\bar X_{t,i} - p_i}$ is large) can cause a large deviation for possibly many other coordinates $\abs{Y_j - p_j}$ in the projection. This makes it hard to apply the argument in Section \ref{subsec:fq_zero} since we need to reason about large deviations of subsets of coordinates.

The second new idea seeks to address this point: we will only project if the sampled average $\bar X_t$ is close enough to the domain $K$. If not, we will resample $\bar X_t$. By doing this, we can guarantee that the coordinates will not move too much.

\subsection{Project if close enough}

The previous discussion motivates the definition a new random variable $Z_{t,\epsilon}$. First, we will define the projection operator $\Pi_K:[0,1]^n \rightarrow K$ which is a function such that:
$$\norm{\Pi_K(p) - p}_\infty \leq \norm{q - p}_\infty, \forall p \in [0,1]^n, q \in K$$
Note that there may be many choices for $\Pi_K$, in which case we may choose arbitrarily. Also observe that $\Pi_K$ is not necessarily continuous. 

Now, define $Y_{t,\epsilon}$ as a random variable taking values in $[0,1]^n$ distributed according to the  conditional distribution of $\bar X_t$ given that $\bar X_t \in \B_\infty(K, \epsilon)$:
$$\P_p [Y_{t,\epsilon} \in A] = \P_p[ \bar X_t \in A \mid \bar X_t \in \B_\infty(K, \epsilon)], \quad \forall \text{ measurable } A \subseteq [0,1]^n$$
This variable can be sampled as follows: first sample $\bar X_t$. If $\bar X_t \in \B_\infty(K, \epsilon)$ then set $Y_{t,\epsilon} = \bar X_t$. If not, resample $\bar X_t$ and try again until $\bar X_t \in \B_\infty(K, \epsilon)$.  Now, define:
$$Z_{t,\epsilon} = \Pi_K(Y_{t,\epsilon})$$

\begin{lemma}\label{lemma:prob_Y}
If $t \geq \log(8n)/(2\epsilon^2)$ then for any $p \in K$ and any measurable set $A$ it holds that:
$$\P_p[Y_{t,\epsilon} \in A] \leq 2\cdot \P_p[\bar X_{t,\epsilon} \in A] $$
\end{lemma}

\begin{proof}
By the definition of $Y_{t,\epsilon}$ we can write:
$$\begin{aligned}
\P_p\left[Y_{t,\epsilon} \in A \right]   = \P_p\left[ \bar X_{t,\epsilon} \in A \middle| \bar X_{t} \in \B_\infty(K, \epsilon)\right]  = \frac{\P_p\left[\bar X_{t,\epsilon} \in A\text{ and } \bar X_{t} \in \B_\infty(K, \epsilon)\right]}{\P_p\left[\bar X_{t} \in \B_\infty(K, \epsilon)\right]} 
\end{aligned}$$
The numerator of the last expression is clearly bounded above by $\P_p[\bar X_{t,\epsilon} \in A]$. For the denominator, notice that for $p \in K$ we have $\B_\infty(p,\epsilon) \subseteq \B_\infty(K,\epsilon) $ hence:
$$\P_p\left[\bar X_{t} \in \B_\infty(K, \epsilon)\right] \geq \P_p\left[\bar X_{t} \in \B_\infty(p, \epsilon)\right] = 1-\P_p\left[\norm{\bar X_{t} -p}_\infty \geq \epsilon \right]  $$
By the Hoeffding bound, we have:
$P_p\left[\norm{\bar X_{t} -p}_\infty \geq \epsilon \right] \leq 2n \exp(-2\epsilon^2 t) \leq \frac{1}{2}$
for $t \geq \log(8n)/(2\epsilon^2)$. Putting it all together, we obtain the result in the statement.
\end{proof}

\subsection{Extension of Lemma \ref{lemma:main_lemma} to subdomains}

\begin{lemma}\label{lemma:sub_main_lemma}
Let $f:K \rightarrow [0,1]$ be a continuous and polynomially bounded function.
Then there is some $\epsilon>0$ and an integer $t_0$ such that for $t \geq t_0$ it holds that:
\begin{equation}\label{eq:sub_main_lemma}
f(p) - \frac{1}{4} \cdot \P_p \left[ f(Z_{t,\epsilon}) \geq \frac{1}{2} \right] \geq \frac{1}{8} f(p), \forall p \in [0,1]^n
\end{equation}
\end{lemma}

With this lemma, the proof of Theorem \ref{thm:sub_main_thm} follows from exact the same arguments used in Section \ref{subsec:sufficient} to prove Theorem \ref{thm:main_thm}. The only new thing to note is that for any fixed $t$ and $\epsilon$ the function $g_{t,\epsilon}(p) =  \P_p \left[ f(Z_{t,\epsilon}) \geq \frac{1}{2} \right]$ can be implemented by a Bernoulli factory since $Z_{t,\epsilon}$ can be sampled with only sample access to the $p_i$-coins.

\subsection{Proof of Lemma \ref{lemma:sub_main_lemma}}

We are now left to prove Lemma \ref{lemma:sub_main_lemma}, for which we will need a slight modification in the arguments. As before we can define:
$$L = \left\{p \in K; f(p) \leq \frac{3}{8} \right\}$$
For $p \notin L$, the statement of the lemma is once again trivial. For $p \in L$ we will follow the strategy in Claim \ref{claim:ball}. First, observe that Claim \ref{claim:buffer} still holds since $L$ and $\left\{p \in K; f(p) \geq \frac{1}{2} \right\}$ are disjoint compact sets. Next we strengthen Claims \ref{claim:zero_open_face} and \ref{claim:safe_distance}:

\begin{claim}\label{claim:sub_safe_distance0}
Let $\bar F_{A,S,B}$ be the closure of $F_{A,S,B}$. If $p \in F_{A,S,B}$ and $f(p) = 0$ then $f\vert_{K \cap \bar F_{A,S,B}} \equiv 0$.
\end{claim}

\begin{proof}
The closure of $F_{A,S,B}$ is the union of all open faces $F_{A',S',B'}$ where $A \subseteq A'$ and $B \subseteq B'$. Now there is some point $q \in K \cap F_{A',S',B'}$ such that $f(q) > 0$ then $f(q) \geq c \cdot ((1-q)^{A'} q^{S'} (1-q)^{S'} q^{B'})^m$ which contradicts the fact that $f(p) = 0$ since $p_i < 1$ for $i \in A'$ since $A' \subseteq A \cup S$,  $p_i > 0$ for $i \in B'$ since $B' \subseteq B \cup S$  and $0 < p_i < 1$ for $i \in S' \subseteq S$.
\end{proof}

\begin{claim}\label{claim:sub_safe_distance}
Let $\bar F_{A,S,B}$ be the closure of $F_{A,S,B}$. If $f\vert_{K \cap \bar F_{A,S,B}} \equiv 0$, there is some $\delta$ such that if $f(q) < \frac{1}{2}$ for all $q \in K \cap B_\infty(\bar F_{A,S,B}, \delta)$.
\end{claim}

\begin{proof}
For the second part, if no such $\delta$ exists, then there must be a sequence of points $q_t$ with $q_t \in K \cap \B_\infty(\bar F_{A,S,B}, \frac{1}{t})$ such that $f(q_t) \geq \frac{1}{2}$. Since $K$ is compact, there must be a subsequence of $q_t$ that converges to a point $q^* \in K \cap \bar F_{A,S,B}$. Since $f(q_t) \geq \frac{1}{2}$ we have also $f(q^*) \geq \frac{1}{2}$. But this contradicts the previous paragraph, which shows that $f(q^*) = 0$.
\end{proof}

Now fix $\delta > 0$ small enough such that $\delta$ satisfies Claims \ref{claim:buffer} and Claim \ref{claim:sub_safe_distance}. We will mirror Sections \ref{subsec:fq_pos} and \ref{subsec:fq_zero} and prove Claim \ref{claim:ball} first for $f(q)>0$ and then for $f(q)=0$. 

\subsubsection{Claim for $f(q)>0$}
We will set $\epsilon = \delta/2$.
Observe that $\norm{Z_{t,\epsilon} - Y_{t,\epsilon}}_\infty \leq \epsilon$ so $\norm{Y_{t,\epsilon} - p}_\infty \leq \norm{Y_{t,\epsilon} - Z_{t,\epsilon}}_\infty + \norm{Z_{t,\epsilon} - p}_\infty \leq \epsilon + \norm{Z_{t,\epsilon} - p}_\infty$, hence:
$$
\P_p\left[ f(Z_{t,\epsilon}) \geq \frac{1}{2}\right] \leq
\P_p\left[\norm{Z_{t,\epsilon} - p}_\infty > \delta\right]
\leq
\P_p\left[\norm{Y_{t,\epsilon} - p}_\infty > \frac{\delta}{2}\right] 
$$
By Lemma \ref{lemma:prob_Y}, the last probability is at most $2\cdot \P_p\left[\norm{\bar X_{t} - p}_\infty > \frac{\delta}{2}\right]$ for large enough $t$. From this point on, the proof is exactly the same as in Section \ref{subsec:fq_pos}.

\subsubsection{Claim for $f(q)=0$}
We set $\epsilon = \delta/4$ and as in Section \ref{subsec:fq_zero}  we split the probability of $\P_p[f(Z_{t,\epsilon}) \geq \frac{1}{2}]$ depending on which components have a large deviation in $Y_{t,\epsilon}$: 
$$\P_p\left[f(Z_{t,\epsilon}) \geq \frac{1}{2}\right] = \sum_{T \subseteq A \cup B} \P_p\left[f(Z_{t,\epsilon}) \geq \frac{1}{2} \text{ and } T = \left\{i \in A \cup B; \abs{(Y_{t,\epsilon})_i - p_i} > \frac{\delta}{2}\right\}\right]$$

As before, we argue that for all points $p$ in a small enough ball around $q$ we have:
\begin{equation}\label{eq:prob_T_sub}
\P_p\left[f(Z_{t,\epsilon}) \geq \frac{1}{2} \text{ and } T = \left\{i \in A \cup B; \abs{(Y_{t,\epsilon})_i - p_i} > \frac{\delta}{2}\right\}\right] \leq \frac{1}{2^{n+2}} f(p)
\end{equation}

For each $T \subseteq A \cup B$ we will consider two cases depending on the value of $f$ on $K \cap \bar F_{A\setminus T, S \cup T, B \setminus T}$. A first observation is that the closure $F_{A\setminus T, S \cup T, B \setminus T}$ contains $F_{A,S,B}$ and hence $K \cap \bar F_{A\setminus T, S \cup T, B \setminus T} \neq \emptyset$.\\

\noindent \emph{Case 1}: $f\vert_{K \cap \bar F_{A\setminus T, S \cup T, B \setminus T}} \equiv 0$. By Claim \ref{claim:sub_safe_distance} we have $f(q) < 1/2$ for all $q \in \B_\infty(\bar F_{A\setminus T, S \cup T, B \setminus T}, \delta)$. Since the distance between $Z_{t,\epsilon}$ and that face is at most $\max_{i \in A\cup B \setminus T}\abs{(Z_{t,\epsilon})_i - (Y_{t,\epsilon})_i} + \abs{(Y_{t,\epsilon})_i -p_i} + r \leq \frac{3\delta}{4}+r < \delta$ for a radius $r< \frac{\delta}{4}$. Hence the probability in equation \eqref{eq:prob_T_sub} is zero.\\

\noindent \emph{Case 2}: $f\vert_{K \cap \bar F_{A\setminus T, S \cup T, B \setminus T}} \not\equiv 0$. In that case, there is an open face $F_{A',S',B'}$ in the closure $\bar F_{A\setminus T, S \cup T, B \setminus T}$ such that $K \cap  F_{A',S',B'} \neq \emptyset$ and $f\vert_{F_{A',S',B'}} \not\equiv 0$. Since $f$ is polynomially bounded, we have:
$$f(p) \geq c \cdot ((1-p)^{A'} p^{S'} (1-p)^{S'} p^{B'})^m \geq c \cdot ((1-p)^{A\setminus T} p^{S \cup T} (1-p)^{S'} p^{B \setminus T})^m $$
where the inequality follows from the fact that
 $A \setminus T \subseteq A'$ and $B \setminus T \subseteq B'$ since $F_{A',S',B'} \subseteq \bar F_{A\setminus T, S \cup T, B \setminus T}$. From this point on, the proof is exactly the same as in Section \ref{subsec:fq_pos}, using Lemma \ref{lemma:prob_Y} to translate statements about $Y_{t,\epsilon}$ to $\bar X_t$.

\section{Extending Sampford Sampling to the Boundary}

Given probabilities $(p_1, p_2, \dots, p_n)$ with $\sum_{i} p_i = k$ (for some integer $1 \leq k < n$), we wish to sample a $k$-element sized subset $U$ of $\{1, 2, \dots, n\}$ with the property that $\P[i \in U] = p_i$. Sampford sampling is a method for accomplishing this given only sample access to coins with these probabilities. Sampford sampling proceeds as follows:

\begin{itemize}
    \item Sample each coin $i$ (with probability $p_i$) once and let $X_i \in \{0, 1\}$ be the outcome.
    \item Let $U = \{i; X_{i} = 1\}$ be the set of coins that came up heads. If $|U| \neq k$, go back to step 1.
    \item Choose a uniform random coin in $[n] \setminus U$, and flip it. If it comes up heads, output $U$. Otherwise, go back to step 1.
\end{itemize}

For each $U \subset [n]$ with $|U| = k$, define

$$g_{U}(x) = \frac{1}{n - k}\left(\prod_{i \in U} x_i\right)\cdot\left(\prod_{i \not\in U} (1-x_i)\right)\cdot\left( \sum_{i \not\in U} x_i\right).$$

Note that $g_{U}(p)$ is exactly the probability that we output a specific set $U$ for one individual trial of the above procedure (i.e., without restarting the procedure). It follows that the above procedure samples a subset $U$ with probability

$$f_{U}(x) = \frac{g_{U}(x)}{\sum_{V \subset [n], |V| = k} g_{U}(x)}.$$

Although $f_{U}(x)$ is defined on the interior $[0, 1]^n$, $f_{U}(x)$ is undefined for some points on the boundary of $[0, 1]^n$ (and even for some points within the subset $K = \{p \in [0, 1]^n ; \sum p_i = k\}$). For example, consider the point $p$ with $p_{i} = 1$ for $1 \leq i \leq k$ and $p_{i} = 0$ for $k + 1 \leq i \leq n$. Although this value of $p$ satisfies $\sum_{i} p_i = k$, $f_{U}(p)$ is undefined at this point; in particular, $g_{U}(p) = 0$ for every single subset $U$. Indeed, for this set of probabilities, it's easy to verify that the procedure described above can never terminate: every round we will sample the set $U = \{1, 2, \dots, k\}$, and then immediately fail the subsequent check in step 3.

In this section we will show that it is indeed possible to construct a multiparameter Bernoulli factory for this problem that terminates almost surely for all valid sets of coins (i.e., the compact subset $K = \{p \in [0, 1]^n ; \sum p_i = k\}$). To do so, we will show that there exists a continuous completion of $f_{U}(x)$ that satisfies the constraints of Theorem \ref{thm:sub_main_thm}.

Given a set $U \subset [n]$, let $e_{U} \in [0, 1]^n$ be the point such that $(e_{U})_i = 1$ if $i \in U$ and $(e_{U})_i = 0$ otherwise. Note that each $e_{U}$ with $|U| = k$ lies in $K$. We show that these are the only points of discontinuity of $f_{U}(x)$ within $K$, and that these discontinuities can be resolved.
\begin{lemma}\label{lemma:fU_continuous}
Define $\overline{f}_{U}(x): K \rightarrow [0, 1]$ as

$$\overline{f}_{U}(x) = \begin{cases}
f_{U}(x) & \text{for } x \in K, x \neq e_{U} \\
1 & \text{for } x = e_{U} \\
0 & \text{for } x = e_{V}, V \neq U.
\end{cases}$$

Then $\overline{f}_{U}(x)$ is a continuous function defined on all of $K$.
\end{lemma}
\begin{proof}
We first show that $f_{U}(x)$ is defined for all points in $K$ not of the form $e_{U}$ (for any subset $U \subset [n]$ of size $k$). To see this, fix an $x \in K$ and let $I(x) = \{i \in [n]; x_i > 0\}$ be the set of non-zero coordinates of $x$. Since $x \in K$, $\sum x_i = k$ and therefore $|I(x)| \geq k$ -- moreover, if $|I(x)| = k$, then we must have $x_{i} = e_{I(x)}$. It follows that if $x$ is not of the form $e_{U}$, then $|I(x)| \geq k+1$.

Now, choose any subset $U'$ of $I(x)$ of size $k$ that contains all indices $i$ such that $x_{i} = 1$ (since $\sum x_i = k$, there are at most $k$ such indices, and they all must belong to $I(x)$). Note that each of the three terms of $g_{U'}(x)$ are positive, so $g_{U'}(x) > 0$. It follows that the denominator of $f_{U}(x)$ is positive, and therefore $f_{U}(x)$ is well-defined for all such points (and therefore $\overline{f}_{U}$ is well-defined for all points in $K$).

It remains to show $\overline{f}_{U}$ is continuous on $K$. It suffices to check continuity at the points $e_{U'}$. To see this, observe that  $f_U(p)$ satisfies $\sum_{U; \abs{U}=k} f_U(p) e_U = p$ for all $p \in K \setminus \{e_U; \abs{U}=k\}$. Now, fix a sequence $p_t \rightarrow e_{U'}$ and some subset $U'' \neq U'$. There is some coordinate $i \in U'' \setminus U'$. Then looking at the $i$-th coordinate we have that: $f_{U''}(p_t) \leq (p_t)_i \rightarrow (e_{U'})_i=0$. Since $f_{U''} \geq 0$ we must have $f_{U''}(p_t) \rightarrow 0$. Hence $\overline{f}_{U''}$ is continuous at all points $e_{U'}$ with $U' \neq U''$. To check continuity when $U' = U''$ take any coordinate $i \in U'$ and $p_t \rightarrow e_{U'}$. Then $f_{U'}(p_t) = (p_t)_i - \sum_{U'' \neq U'} f_{U''}(p_t) (e_{U''})_i \rightarrow 1$ by the previous observation. Hence $\overline{f}_{U'}$ is also continuous at $e_{U'}$.

\end{proof}



We will now show that this function $\overline{f}_{U}$ satisfies the constraints of Theorem \ref{thm:sub_main_thm} on the set $K$ (and hence we can construct a multivariate Bernoulli factory for the function $\overline{f}_{U}$ that terminates a.s. for all points in $K$).

\begin{lemma}\label{lemma:fU_poly_bounded}
The function $\overline{f}_{U}$ is polynomially bounded on $K$.
\end{lemma}
\begin{proof}
Let us begin by characterizing the faces $F_{A, S, B}$ where there exists a point $p \in F_{A, S, B} \cap K$ such that $\overline{f}_{U}(p) > 0$. In particular, we claim that if this happens, then $B \subseteq U \subseteq (S \cup B)$. To see why, note that if $i \in B$ then $p_i = 1$ for $p \in F_{A, S, B}$, so if we have a positive probability of outputting subset $U$, $U$ must contain element $i$. Similarly, if $i \in A$, then $p_i = 0$ for $p \in F_{A, S, B}$, so if we have a positive probability of outputting subset $U$, $U$ cannot contain element $i$ (and thus must be contained in $S \cup B$). 

We will now show that for all $p \in K$ and faces $F_{A, S, B}$ satisfying $B \subseteq U \subseteq (S \cup B)$, there exist constants $c, m > 0$ such that

\begin{equation}\label{eq:sampford_polynomial_bounded}
    \overline{f}_{U}(p) \geq c \cdot ((1-p)^{A} \cdot (1-p)^{S}p^{S} \cdot p^{B})^{m}.
\end{equation}

Let $\overline{U} = [n] \setminus U$. Note that since $p, 1-p \leq 1$, $U \subseteq (S \cup B)$, and $\overline{U} \subseteq (S \cup A)$, \eqref{eq:sampford_polynomial_bounded} is implied by the following inequality:

\begin{equation}\label{eq:sampford_poly_bounded2}
    \overline{f}_{U}(p) \geq c\left((1-p)^{\overline{U}} \cdot p^{U}\right)^{m}.
\end{equation}

We will prove \eqref{eq:sampford_poly_bounded2}. First, note that this holds for all $p$ of the form $e_{V}$ with $|V| = k$ (in particular, whenever $\overline{f}_{U}(e_{V}) = 0$, the RHS of \eqref{eq:sampford_poly_bounded2} is also $0$). It suffices to prove \eqref{eq:sampford_poly_bounded2} on all other points of $K$. On these points $\overline{f}_{U}(p) = f_{U}(p)$, so by substituting in the definition of $f_{U}(p)$, it suffices to prove that

\begin{equation}\label{eq:sampford_poly_bounded3}
    g_{U}(p) \geq c\left((1-p)^{\overline{U}} \cdot p^{U}\right)^{m} \sum_{|V| = k} g_{V}(p).
\end{equation}

Note that

$$g_{U}(p) = \frac{1}{n-k} p^{U} (1-p)^{\overline{U}} \sum_{i\in U} p_i.$$

Inequality \eqref{eq:sampford_poly_bounded3} thus reduces to

\begin{equation}\label{eq:sampford_poly_bounded4}
    \sum_{i\not\in U}p_i \geq c(n-k)\left((1-p)^{\overline{U}} \cdot p^{U}\right)^{m - 1} \sum_{|V| = k} g_{V}(p).
\end{equation}

We will now prove the following: for any $\eps > 0$, if $\sum_{i \not\in U} p_{i} = \eps$, then $g_{V}(p) \leq k\eps$ for each $V \subseteq [n]$ with $|V| = k$. Note that this implies \eqref{eq:sampford_poly_bounded4} (in particular, it suffices to set $c = 1/((n-k)k\binom{n}{k})$ and $m = 1$). 

To show the above claim, note that if $\sum_{i \not\in U} p_{i} = \eps$, then $p_{i} \leq \eps$ for all $i \not\in U$. Now, note that for any $V \neq U$ with $|V| = |U| = k$, there must exist an index $i^*$ belonging to $V$ that does not belong to $U$. It follows that for $V \neq U$.

$$g_{V}(p) = \frac{1}{n-k}p^{V}(1-p)^{\overline{V}} \left( \sum_{i\in V} p_i \right) \leq kp_{i^*} \leq k \eps.$$

For $V = U$, it immediately follows that $g_{U}(p) \leq \sum_{i \in U} p_i = \eps$, concluding our proof. 
\end{proof}

\begin{lemma}\label{lemma:sum_poly_bounded}
If functions $f_1, \hdots, f_k : K \rightarrow [0,1]$ are polynomially bounded, then so is their sum $f := f_1 + \hdots + f_k$.
\end{lemma}

\begin{proof}
If for a certain open face $F_{A,S,B}$ we have $f \vert_{K \cap F_{A,S,B}} \not\equiv 0$ then there must one one index $i$ such that $f_i \vert_{K \cap F_{A,S,B}} \not\equiv 0$ and since $f_i$ is polynomially bounded, we have $f(p) \geq f_i(p) \geq c ((1-p)^A p^S (1-p)^S p^B)^m$ for some constant $c > 0$ and integer $m \geq 0$.
\end{proof}

\begin{lemma}[Bernoulli race (\cite{dughmi2017bernoulli})]\label{lemma:bernoulli_race}
If functions $f_1,\hdots, f_k:K \rightarrow [0,1]$ can be implemented by a Bernoulli factory and $\sum_{j=1}^k f_j(p) > 0$ for all $p \in K$, then there exists a sampling algorithm that for each $p \in K$ samples an index $i \in [k]$ with probability proportional to $f_i(p)/(\sum_{j=1}^k f_j(p))$
\end{lemma}

\begin{proof}
We sample an index $i \in [k]$ uniformly at random and then sample from the $f_i(p)$-coin. If it comes up $1$ we return index $i$. Otherwise we retry. The procedure terminates a.s. since it has a positive probability of outputting for each retry. Since each trial outputs index $i$ with probability proportional to $f_i(p)/k$, the overall procedure outputs index $i$ with probability $f_i(p)/(\sum_{j=1}^k f_j(p))$.
\end{proof}

\begin{theorem}
There exists a multiparameter factory for Sampford sampling that terminates everywhere in the set $K = \{p \in [0,1]^n ; \sum p_i = k \}$. 
\end{theorem}
\begin{proof}
We first argue that $\overline{f}_{U}$ satisfies the constraints of Theorem \ref{thm:sub_main_thm} for the set $K$, and therefore that we can construct a multivariate Bernoulli factory for $\overline{f}_U$ that terminates a.s. for all $p \in K$.

To show this, we must show that $\overline{f}_{U}$ is continuous on $K$, and that both $\overline{f}_{U}$ and $1 - \overline{f}_{U}$ are polynomially bounded on $K$. We have already shown that $\overline{f}_{U}$ is continuous on $K$ (Lemma \ref{lemma:fU_continuous}) and that $\overline{f}_{U}$ is polynomially bounded on $K$ (Lemma \ref{lemma:fU_poly_bounded}). To see that $1 - \overline{f}_{U}$ is polynomially bounded on $K$, note that $1 - \overline{f}_{U} = \sum_{V \neq U} \overline{f}_{V}$. Since this is a sum of functions each polynomially bounded on $K$, it follows that $1 - \overline{f}_{U}$ is polynomially bounded on $K$ (Lemma \ref{lemma:sum_poly_bounded}). 

Finally, we will use our factories that output a coin with probability $\overline{f}_{U}(p)$ to construct a factory that outputs an actual subset $U$ using the Bernoulli race in Lemma \ref{lemma:bernoulli_race}.

\end{proof}

\section{Combinatorial Bernoulli Factories}

Given a polytope $\PP \subseteq [0, 1]^n$ (with vertices $V(\PP)$), a \textit{combinatorial Bernoulli factory} for $\mathcal{P}$ is an exact sampling procedure that, given coins $(p_1, p_2, \dots, p_n) \in \mathcal{P}$, outputs a vertex $v \in V(\mathcal{P})$ such that $\E_{p}[v] = p$. More formally, a combinatorial Bernoulli factory is a collection\footnote{The Bernoulli race in Lemma \ref{lemma:bernoulli_race} is used to convert this collection of factories into a procedure for sampling a vertex with the desired probability.
} of $|V(\mathcal{P})|$ multiparameter Bernoulli factories for functions $f_{v}(p)$ satisfying (for all $p \in \mathcal{P}$):

$$\sum_{v \in V(\mathcal{P})} f_{v}(p) = 1 \quad \text{ and } \quad \sum_{v \in V(\mathcal{P})} vf_{v}(p) = p.$$

Combinatorial Bernoulli factories capture a wide range of combinatorial sampling problems. For example, the problem of Sampford sampling is equivalent to the problem of constructing a combinatorial Bernoulli factory for the polytope $\PP = [0, 1]^n \cap \{ p \in \R^n | \sum_{i}p_i = k\}$. Other problems captured by combinatorial Bernoulli factories include exact sampling of matchings and flows.

\cite{NiazadehLS21} show that any polytope $\PP$ that admits a combinatorial Bernoulli factory must be of the form $\PP = [0, 1]^n \cap K$, where $K$ is an affine subspace of $\R^n$. Moreover, they give a general method for constructing combinatorial Bernoulli factories for any such polytope; however, as with existing implementations of Sampford sampling, the factories they generate can fail to terminate at some points on the boundary of $[0, 1]^n$. In this section, we provide an alternate method for constructing combinatorial Bernoulli factories that works for all polytopes of the form $\PP = [0, 1]^n \cap K$, \emph{everywhere} in $\PP$.

We begin by presenting our new construction. Given a $d$-dimensional polytope $\PP$ and a vertex $w$ of $\PP$, we say that the \textit{fan triangulation} $\mathcal{T}_w$ of $\PP$ corresponding to vertex $w$ is the division of $\PP$ into simplices with disjoint interiors formed by connecting $w$ to each facet of $\PP$ that does not contain $w$ (if a facet contains more than $d$ vertices of $\PP$, arbitrarily triangulate it into $(d-1)$-dimensional simplices first).

Note that given a simplex, there is a unique way to write a point in the simplex as a convex combination of its vertices. This implies that any triangulation $\mathcal{T}$ of $\PP$ gives rise to a natural way to decompose a point $p \in \PP$ as a convex combination of the vertices of $\PP$: namely, find the simplex $T$ of the triangulation that $p$ belongs to, and write $p$ as a convex combination of the vertices of $T$. Let $g^{(w)}_{v}(p)$ be the coefficient of vertex $v$ in the decomposition stemming from the fan triangulation $\mathcal{T}_w$. Note that all these functions $g^{(w)}_v : \PP \rightarrow [0,1]$ are continuous since they are continuous on each simplex of the triangulation and agree on the common faces.

\begin{equation}\label{eq:decomp_fac}
f_{v}(p) = \frac{1}{|V(\PP)|} \sum_{w \in V(\PP)} g^{(w)}_{v}(p).
\end{equation}

By construction, it follows that $\sum_{v}f_{v}(p) = 1$ and $\sum_{v} vf_{v}(p) = p$ for all $p \in \PP$. In the remainder of the section, we will show that if $\PP$ is of the form $[0, 1]^n \cap K$ for some affine subspace $K$, then each $f_{v}(p)$ satisfies the conditions of Theorem \ref{thm:sub_main_thm} (and thus can be implemented by a multiparameter Bernoulli factory).\\

We will need the following lemma.

\begin{lemma}\label{lem:comb_geometry}
Let $\PP$ be a polytope of the form $[0, 1]^n \cap K$, where $K$ is an affine subspace of $\R^n$. Let $v \in V(\PP)$ be a vertex of $\PP$, and let $F$ be a facet of $\PP$ that doesn't contain $v$. Then there exists a coordinate $i$ such that either $v_i > 0$ and $x_i = 0$ for all $x \in F$, or $v_i < 1$ and $x_i = 1$ for all $x \in F$.
\end{lemma}
\begin{proof}
Since $K$ is an affine subspace, each facet of $\PP$ can be written as the intersection of a facet of $[0, 1]^n$ with $K$. Each facet of $[0, 1]^n$ is given by a single constraint of the form $x_i = 0$ or $x_{i} = 1$. 

Assume that the facet $F$ is equal to $\{x_i = 0\} \cap K$. Then if $v$ is not contained in $F$, it must be the case that $v_i \neq 0$ (and thus $v_i > 0$ and the first condition of the theorem holds). Similarly, if the facet $F$ is given by $\{x_i = 1\} \cap K$, then any vertex $v$ not contained in $F$ must satisfy $v_i < 1$, and the second condition of the theorem holds.
\end{proof}

Note that the previous lemma fails if $\PP$ is not the intersection of the hypercube with an affine subspace. For example, if $\PP$ is the convex hull of $(0,0), (1,0), (0,1)$, the lemma fails for $v=(0,0)$ and $F$ the opposite edge. This is an important sanity check, as \cite{NiazadehLS21} shows that no other polytope admits a combinatorial Bernoulli factory.

We can now show that $f_{v}(p)$ is polynomially bounded (and thus that we can construct combinatorial Bernoulli factories that terminate everywhere on $\PP$).

\begin{lemma}\label{lem:f_poly_bounded}
If $\PP$ is a polytope of the form $[0, 1]^n \cap K$ where $K$ is an affine subspace of $\R^n$, then the functions $f_{v}(p)$ are polynomially bounded on $\PP$.
\end{lemma}
\begin{proof}
Fix a vertex $v$ of $\PP$. To begin, we'll argue that $g^{(v)}_{v}(p) > 0$ for exactly the points $p \in \PP$ where $f_{v}(p) > 0$. To see this, note that if $g^{(v)}_{v}(p) > 0$, then $f_{v}(p) > 0$ (since $g^{(v)}_{v}(p)$ is a summand in $f_{v}(p)$). But conversely, by construction $g^{(v)}_{v}(p)$ only equals $0$ on (closed) faces of $\PP$ that do not contain $v$. Since $f_{v}(p)$ also forms a convex decomposition of $p$ into the vertices of $\PP$, $f_{v}(p)$ must also equal $0$ on all these closed faces and it follows that $g_{v}(p) = 0$ implies that $f_{v}(p) = 0$.

We will now show that $g^{(v)}_{v}(p)$ is polynomially bounded on $\PP$; it then follows from \eqref{eq:decomp_fac} that $f_v(p)$ is polynomially bounded on $\PP$ (since $f_v(p)>0$ implies $g_v^{(v)}(p) > 0$). 

Recall that $g^{(v)}_{v}(p)$ is the decomposition induced by the fan triangulation $\mathcal{T}_v$. That is, to compute the value of $g^{(v)}_{v}(p)$, we first must identify the simplex of $\mathcal{T}_v$ that $p$ belongs to, and (uniquely) write $p$ as a convex combination of the vertices of that simplex.

Let us assume that $p$ belongs to the simplex $T \in \mathcal{T}_v$. Since $\mathcal{T}_v$ is the fan triangulation for vertex $v$, $T$ must be the convex hull of a facet $F$ of $\PP$ (not containing $v$) and $v$. By Lemma \ref{lem:comb_geometry}, there exists some coordinate $i$ such that either $v_i > 0$ and $F \subset \{x; x_i = 0\}$ or $v_{i} < 1$ and $F \subset \{x; x_i = 1\}$. 

In the first case, note that $g^{(v)}_{v}(p)$ must equal $p_i/v_i$ (since $v_i$ is the only vertex of $T$ that contains a positive $i$th component). We claim that given this, $g^{(v)}_{v}(p)$ is polynomially bounded. To see why, note that if $f_v(p)>0$ then $g^{(v)}_{v}(p) > 0$ and as a consequence $p_i > 0$ (since $v_i > 0$). It follows that if $g^{(v)}_{v}(p) > 0$ and $p$ belongs to some open face $F_{A, S, B}$ of the hypercube, then $i \not\in A$. But now, note that if $i \not\in A$, then there exist constants $c$ and $m$ such that

\begin{equation}\label{eq:decomp_poly_bounded}
\frac{p_i}{v_i} \geq c \cdot ((1-p)^{A}\cdot p^S(1-p)^S\cdot p^B)^{m}.
\end{equation}

In particular, since $i$ lies either in $S$ or $B$, it suffices to take $m = 1$ and $c = 1/v_i$. 

Similarly, in the second case $g^{(v)}_{v}(p)$ must equal $(1-p_i)/(1-v_i)$. A similar argument shows that $g^{(v)}_{v}(p)$ is polynomially bounded in this case (now we must have $i \not\in B$, and a factor of $(1-p_i)$ will appear on the RHS of the analogue of \eqref{eq:decomp_poly_bounded}).
\end{proof}

\begin{theorem}
If $\PP$ is a polytope of the form $[0, 1]^n \cap K$, where $K$ is an affine subspace $\R^n$, then there exists a combinatorial Bernoulli factory for $\PP$ which terminates almost surely everywhere on the boundary.
\end{theorem}
\begin{proof}
It suffices to show that the functions $f_{v}(p)$ defined in \eqref{eq:decomp_fac} satisfy the conditions of Theorem \ref{thm:sub_main_thm}. Namely, we must show that $f_{v}(p)$ are continuous, and that both $f_{v}(p)$ and $1 - f_{v}(p)$ are polynomially bounded on $\PP$. 

Since each $g^{(w)}_v(p)$ is continuous, $f_{v}(p)$ is continuous. By Lemma \ref{lem:f_poly_bounded}, each $f_{v}(p)$ is polynomially bounded on $\PP$. Finally, note that $1 - f_{v}(p) = \sum_{w \neq v} f_{w}(p)$, so $1-f_{v}(p)$ is polynomially bounded by Lemma \ref{lemma:sum_poly_bounded}.  
\end{proof}

\bibliographystyle{plainnat}
\bibliography{bernoulli}

\appendix

\section{Missing Proofs}

\begin{proof}[Proof of Lemma \ref{lemma:necessary3}]
Consider an implementation of $f$ by a Bernoulli factory and fix a point $a \in [0,1]^n$ in the domain of $f$. We want to show that for every $\epsilon > 0$, there is $\delta$ such that if $\norm{p-a} < \delta$ then $\abs{f(p) - f(a)} < \epsilon$. 

To show that, let $T$ be a random variable showing the number of coins flipped before the output if the decision tree is executed using an $a$-coin (this is equal to the depth of the output node reached in the tree). Now, fix $t$ such that $\P_a[T > t] < \epsilon/4$. Represent a possible realization of the first $t$ coin flips of each coin by a tuple $x=(x_1,\hdots, x_t)$ for $x_i \in \{0,1\}^n$ we define function $F(x) \in \{0,1,\emptyset\}$ indicating whether the decision tree outputs $0$, $1$ or doesn't yet terminate after seeing inputs $x_1, \hdots, x_t$. Also, let $X=(X_1, \hdots, X_t) \in \{0,1\}^{nt}$ be the random output of the coins.  With that, we can rewrite $\P_a[T > t] < \epsilon/4$ as:
\begin{equation}\label{eq:prob_output}
\sum_{x \in \{0,1\}^{nt}; F(x) = \emptyset} \P_a[X=x] \leq \frac{\epsilon}{4}
\end{equation}

Now, choose $\delta$ small enough such that the total variation distance between the sequences $X = (X_1, \hdots, X_{t})$ generated under $p$ and $a$ is at most $\epsilon/3$ for any $\norm{p-a} < \delta$. More formally:
\begin{equation}\label{eq:tv_distance}
\sum_{x \in \{0,1\}^{nt}} \abs{\P_{a}[X = x ] - \P_{p}[X = x ]} < \frac{\epsilon}{4}, \forall p \in  \B_\infty(a,\delta)
\end{equation}
Now, we can bound $f(a)$ and $f(p)$ for $\norm{p-a} < \delta$ as follows:
$$\left|f(a) - \sum_{x \in \{0,1\}^{nt}; F(x) \in \{0,1\}} F(x) \P_a[X=x] \right| \leq \sum_{x \in \{0,1\}^{nt}; F(x) = \emptyset} \P_a[X=x] < \frac{\epsilon}{4}$$
and similarly:
$$\left|f(p) - \sum_{x \in \{0,1\}^{nt}; F(x) \in \{0,1\}} F(x) \P_{p}[X=x] \right| \leq \sum_{x \in \{0,1\}^{nt}; F(x) = \emptyset} \P_{p}[X=x] < \frac{\epsilon}{2}$$
where the last bound follows from combining equations \eqref{eq:prob_output} and \eqref{eq:tv_distance}. Now, taking it all together, we have:
$$\abs{f(a) - f(p)} \leq \left|\sum_{x \in \{0,1\}^{nt}; F(x) \in \{0,1\}} F(x) (\P_a[X=x] - \P_{p}[X=x]) \right| + \frac{3\epsilon}{4} < \frac{\epsilon}{4} + \frac{3\epsilon}{4} = \epsilon $$
\end{proof}

\end{document}